\documentclass[11pt,reqno]{amsart}

\usepackage{amssymb}
\usepackage[all]{xy}
\catcode`\@=11

\long\def\@savemarbox#1#2{\global\setbox#1\vtop{\hsize\marginparwidth 
  \@parboxrestore\tiny\raggedright #2}}
\newcommand\lref[1]{\ref{#1}%
\@ifundefined{r@DisplaY #1}{}{ (#1)}}% Prints label as well as
\newcommand\fakelabel[2]{\@bsphack\if@filesw {\let\thepage\relax
   \newcommand\protect{\noexpand\noexpand\noexpand}%
\xdef\@gtempa{\write\@auxout{\string
      \newlabel{#1}{{#2}{\thepage}}}}}\@gtempa
   \if@nobreak \ifvmode\nobreak\fi\fi\fi\@esphack}

\catcode`\@=12

\def\Empty{}
\newcommand\oplabel[1]{
  \def\OpArg{#1} \ifx \OpArg\Empty {} \else
        \label{#1}
  \fi}
\newtheorem{theoremSt}{Theorem}[section]

\newtheorem{exampleSt}[theoremSt]{Example}
\newtheorem{exerciseSt}[theoremSt]{Exercise}

\newcommand\MakeStEnv[1]{
  \newenvironment{#1}[1]{%    environment without explicit label
  \begin{#1St} \oplabel{##1}%
  \global\def\CrntSt{\thetheoremSt}%
}{ 
  \end{#1St} }
  \newenvironment{#1+}[1]{%   environment with explicit label
  \begin{#1St} \label{##1}%
  \label{DisplaY ##1}%
  \global\def\CrntSt{\thetheoremSt}%
  \def\Labl{##1}\ifx\Labl\Empty{} \else {\em (\Labl)\,}\fi%
}{ 
  \end{#1St} }
}
\MakeStEnv{theorem}
\MakeStEnv{corollary}
\MakeStEnv{proposition}
\MakeStEnv{lemma}
\MakeStEnv{definition}
\MakeStEnv{conjecture}
\MakeStEnv{problem}
\MakeStEnv{question}

\newlength{\saveu}

\newenvironment{pf*}[1]{%
 \begin{proof}[#1]%
}{ 
 \end{proof}
}

\newcommand{\finishproof}[1]{ 
  \def\FPArg{#1}
  \ifx\FPArg\Empty
        \newcommand\FPArg{\CrntSt}  \fi
  \smallbreak\noindent\makebox[\textwidth]{\hfill\fbox{\FPArg}}
  \medbreak\noindent
}

\newcommand\CC{{\mathcal C}}

\newcommand\FF{{\mathcal F}}
\newcommand\GG{{\mathcal G}}

\newcommand\LL{{\mathcal L}}
\newcommand\MM{{\mathcal M}}
\newcommand\NN{{\mathcal N}}

\newcommand\PP{{\mathcal P}}
\newcommand\QQ{{\mathcal Q}}

\newcommand\UU{{\mathcal U}}

\newcommand\PMF{{\PP\kern-2pt\MM\FF}}

\newcommand\PML{{\PP\kern-2pt\MM\LL}}

\newcommand\ep{\epsilon}

\newcommand\hhat{\widehat}

\newcommand\union{\cup}
\newcommand\intersect{\cap}
\newcommand\bbR{{\mathord{\text{I\kern-2pt R}}}}        % Fake blackboard
\newcommand\bbH{{\mathord{\text{I\kern-2pt H}}}}        % Fake blackboard

\newcommand\C{{\mathbb C}}

\newcommand\Z{{\mathbb Z}}

\newcommand\N{{\mathbb N}}

\newcommand\bigrightarrow[1]{\hbox to #1{\rightarrowfill}}
\newcommand\bigleftarrow[1]{\hbox to #1{\leftarrowfill}}

\newcommand\boundary{\partial}
\newcommand\semidir{\mathrel{\hbox{\vrule depth-.03ex height1.1ex\kern-0.15em$\times$}}}

\newcommand{\diam}{\operatorname{diam}}

\numberwithin{equation}{section}

\catcode`\@=11

\def\subsection{\@startsection{subsection}{2}%
  \z@{.5\linespacing\@plus.7\linespacing}{.5em}%
  {\normalfont\bfseries\centering}}

\def\section{\@startsection{section}{1}%
  \z@{.7\linespacing\@plus\linespacing}{.5\linespacing}%
  {\normalfont\large\bfseries\centering}}

\def\subsubsection{\@startsection{subsubsection}{3}%
  \z@{.5\linespacing\@plus.7\linespacing}{-.5em}%
  {\normalfont\bfseries}}

\catcode`\@=12
\newcommand{\fsubd}{\mathrel{{\scriptstyle\searrow}\kern-1ex^d\kern0.5ex}}
\newcommand{\bsubd}{\mathrel{{\scriptstyle\swarrow}\kern-1.6ex^d\kern0.8ex}}
\newcommand{\fsubeq}{\mathrel{\raise-.7ex\hbox{$\overset{\searrow}{=}$}}}
\newcommand{\bsubeq}{\mathrel{\raise-.7ex\hbox{$\overset{\swarrow}{=}$}}}

\newcommand{\tsh}[1]{\left\{\kern-.9ex\left\{#1\right\}\kern-.9ex\right\}}
\newcommand{\Tsh}[2]{\tsh{#2}_{#1}}

\def\MCG{\mathcal {MCG}}

\def\co{\colon}

\def\diam{{\rm{diam}}}

\newcommand\truncate[2]{\Tsh{#2}{#1}}

\renewcommand\MCG{\mathcal{MCG}}

\def\hull{\operatorname{hull}}

\title[Centroids and rapid decay 
in mapping class groups]{Centroids and the rapid decay property\\ 
in mapping class groups}
\author{Jason A. Behrstock}
\address{Lehman College, CUNY}
\email{jason.behrstock@lehman.cuny.edu}
\author{Yair N. Minsky}
\address{Yale University}
\email{yair.minsky@yale.edu}

\thanks{Partially supported by NSF grants DMS-0812513 and DMS-0504019}

\begin{document}

\begin{abstract}
We study a notion of an equivariant, Lipschitz, permutation-invariant \emph{centroid} for triples of points in
mapping class groups $\MCG(S)$, which satisfies a certain polynomial growth bound.
A consequence (via work of Dru\c tu-Sapir or Chatterji-Ruane) is the Rapid Decay 
Property for $\MCG(S)$. 
\end{abstract}

\maketitle

%%\setcounter{tocdepth}{1}
%%\tableofcontents

\section{Introduction}

A finitely generated group has the 
\emph{Rapid Decay property}\footnote{This property is also sometimes 
called \emph{property RD} or the \emph{Haagerup inequality}} 
if the 
space of rapidly decreasing functions on $G$ (with respect to every  
word metric) is inside the reduced $C^{*}$--algebra of $G$ (see the
end of 
section \ref{background} for a more detailed definition). Rapid Decay
was first introduced for the free group by Haagerup 
\cite{Haagerup}. Jolissaint then formulated this property in  
its modern form and established it for several classes of groups, 
including 
groups of polynomial growth and discrete cocompact subgroups of 
isometries of hyperbolic space \cite{Jolissaint:RD}. Jolissaint also 
showed that many groups, for instance $SL_{3}(\Z)$, fail to have the Rapid 
Decay property \cite{Jolissaint:RD}. 
Rapid Decay was 
established for Gromov-hyperbolic groups by de la Harpe 
\cite{delaHarpe:hyperbolicRD}.

Throughout this paper $S=S_{g,p}$ will denote a compact orientable 
surface with genus $g$ and $p$ punctures. The mapping class group of 
$S$, denoted $\MCG(S)$, is the group of isotopy classes of 
orientation preserving homeomorphisms of $S$. We will prove:

\begin{theorem}{rapid decay}
  $\MCG(S)$ has the Rapid Decay property for every compact orientable 
  surface $S$.
\end{theorem}
\noindent 
The only previously known cases of this theorem were in low
complexity when the mapping class group is hyperbolic (these are tori 
with at most one puncture, or spheres with at most 4 punctures) 
and for the
braid group on four strands, which was recently established by Barr\'e
and Pichot \cite{BarrePichot:4strandbraid}. 
The results in this paper 
also hold 
%in the case of surfaces with boundary, which in particular includes 
for the braid 
group on any number of strands. (The case of braid groups follows from 
the above theorem, since 
braid groups are subgroups of mapping class groups of surfaces 
\cite[Theorem 2.7.I]{Ivanov:mcg} and  
RD is inherited by subgroups 
\cite[Proposition 2.1.1]{Jolissaint:RD}.)

The Rapid Decay property has several interesting 
applications. For instance, in order to prove the Novikov Conjecture for 
hyperbolic groups, Connes-Moscovici 
\cite[Theorem~6.8]{ConnesMoscovici:hyperbolic} showed that if a 
finitely generated group has the Rapid Decay property and
has group cohomology of polynomial growth (property PC), 
then it satisfies Kasparov's Strong Novikov Conjecture 
\cite{Kasparov:Novikov}. Accordingly, since any 
automatic group has property PC \cite{Meyer:combable},  
Kasparov's Strong Novikov Conjecture follows from the above 
Theorem~\ref{rapid decay} and Mosher's result that mapping 
class groups are automatic \cite{mosher:automatic}. 
The strong Novikov conjecture for $\MCG(S)$ has been previously
established by both Hamenst\"adt  
\cite{Hamenstadt:boundaryamenability} and 
Kida \cite{Kida:MCGmeasureequivviewpoint}.

\medskip

We prove the Rapid Decay property by appealing to a reduction by
Dru\c{t}u-Sapir (alternatively Chatterji-Ruane) to a geometric
condition. Namely, we introduce a notion of 
\emph{centroids} for unordered triples in the mapping class group which satisfies 
a certain polynomial growth property. Despite the presence of large 
quasi-isometrically embedded flat subspaces in the mapping class 
group, these centroids behave much like centers of triangles 
in hyperbolic space. 
Our notion of centroid is provided by the following result 
which to each unordered triple in the mapping class group 
gives a Lipschitz assignment of a point, which has the property that 
it is a centroid in every curve complex projection. We obtain the 
following:

\begin{theorem}{centroid}
For each $S=S_{g,p}$ with $\xi(S) = 3g-3+p \ge 1$
there exists a map $\kappa\co\MCG(S)^3 \to \MCG(S)$ with the following
properties: 
\begin{enumerate}
\item \label{centroid:perm} $\kappa(x,y,z)$ is invariant under permutation of the arguments.
\item \label{centroid:equivariant} $\kappa$ is equivariant.
\item $\kappa$ is Lipschitz.
\item \label{centroid:count} For any $x,y\in\MCG(S)$ and $r>0$ we have 
the following cardinality bound:
$$
\#\left\{\kappa(x,y,z): d(x,z) \le r\strut\right\} \le  b r^{\xi(S)}
$$
where $b$ depends only on $S$. 
\end{enumerate}
\end{theorem}

These properties, and especially the count provided by part (4), are
essentially Dru\c{t}u and Sapir's 
condition of {\em  (**)-relative hyperbolicity} with respect to the 
trivial subgroup \cite{DrutuSapir:RD}, 
and the main theorem of \cite{DrutuSapir:RD} states that this 
condition implies the Rapid Decay property. 
Thus to obtain Theorem~\ref{rapid decay} from Theorem \ref{centroid} 
we appeal to \cite{DrutuSapir:RD}, without dealing directly 
with the Rapid Decay property itself.

\subsection*{Outline of the proof}

Let us first recall the situation for a hyperbolic group, $G$. In 
this setting, for a triple of points $x,y,z\in G$, one defines a 
centroid for the triangle with vertices $x,y,z$ to be a point 
$\kappa$ with the property that $\kappa$ is in the 
$\delta$--neighborhood of any geodesics $[x,y],[y,z],[x,z]$, where 
$\delta$ is the hyperbolicity constant for $G$ considered with some 
fixed word metric. Thus, if 
one fixes $x$ and $y$ and allows $z$ to vary in the  
ball of radius $r$ around $x$, the 
corresponding centroid must lie in a $\delta$--neighborhood of the 
length $r$ initial segment of $[x,y]$. It follows that the 
number of such centers is linear in $r$.

When $3g+p-3>1$, then $\MCG(S_{g,p})$ is not hyperbolic. 
Nonetheless, it has a closely associated space, 
the complex of curves, $\CC(S)$, which is hyperbolic 
\cite{MasurMinsky:complex1}. Moreover, given any subsurface $W\subseteq S$, 
there is a geometrically defined \emph{projection map}, $\pi_{W}$, 
from the mapping class group of $S$ to the curve complex of $W$.

For any $x,y,z\in\MCG(S)$, in Theorem~\ref{centroid from 
projections}, we construct a centroid $\kappa(x,y,z)$ 
with the property that for each $W\subset S$, in the hyperbolic space
$\CC(W)$ the point $\pi_W(\kappa(x,y,z))$ is a centroid of the 
triangle with vertices $\pi_W(x)$, $\pi_W(y)$, and $\pi_W(z)$.

Due to the lack of hyperbolicity in $\MCG(S)$, if one were to fix 
ahead of time a geodesic $[x,y]$, it need not be the case that the 
center $\kappa(x,y,z)$ is close to $[x,y]$. For this reason, we do 
not fix a geodesic between $x$ and $y$, but rather we use the 
notion of a $\Sigma$--hull, as introduced in \cite{BKMM}. The 
$\Sigma$--hull of a finite set is a way of taking the convex 
hull of these points, in particular, the convex hull of a pair of 
points is roughly the union of all geodesics between those points.

In analogy to the fact that for any triangle in a 
Gromov-hyperbolic space any centroid  
is uniformly close to each of the three geodesics, in 
Section \ref{polynomial bounds} we show that in $\MCG(S)$, 
any centroid $\kappa(x,y,z)$ is contained in each  
$\Sigma$--hull between a pair of vertices. 
This reduces the problem of counting centroids 
to counting subsets of the $\Sigma$--hull, which we also do in this 
section.

\medskip

In Section~\ref{background}, we will review the relevant properties 
of surfaces, curve complexes, and mapping class groups. 
In Section \ref{centroid construction}, we will use properties of curve complexes and
$\Sigma$--hulls, as developed in \cite{BKMM}, to construct the Lipschitz,
permutation-invariant centroid
map. In Section \ref{polynomial bounds} we will prove the polynomial bound (3), thus
completing the proof of Theorem \ref{centroid}.

\subsubsection*{Acknowledgements}

The authors would like to thank Indira Chatterji, 
Cornelia Dru\c{t}u, and Mark Sapir for 
useful conversations and for raising our interest in the Rapid Decay 
property. We would like to thank Ken Shackleton for comments on an  
earlier draft. 
Behrstock would also like to thank the 
Columbia University Mathematics Department for their support.

\section{Background}
\label{background}

We recall first some notation and results that were developed in
\cite{MasurMinsky:complex1}, \cite{MasurMinsky:complex2} and \cite{BKMM}. 

\subsubsection*{Surfaces and subsurfaces}
As above, $S=S_{g,p}$ is an oriented
connected surface with genus $g$ and $p$ punctures 
(or boundary components) and 
we measure the complexity of this surface by 
$\xi(S_{g,p}) = 3g-3+p$. An
essential subsurface $W\subseteq S$ is one whose inclusion is
$\pi_1$--injective, and which is not peripheral, i.e., not homotopic to
the boundary or punctures of $S$.  We also consider disconnected
essential subsurfaces, in which each component is essential and no two
are isotopic. For such a subsurface $X$ we define another notion of
complexity $\xi'(X)$ as follows: $\xi'(X) = \xi(X)$ if $X$ is
connected and $\xi(X)\ge 0$, $\xi'(Y)=1$ if $Y$ is an annulus, and
$\xi'$ is additive over components of a disconnected surface. (In
\cite{BehrstockMinsky:rankconj}, $\xi'(S)$ was denoted $r(S)$).  It is
not hard to check that $\xi'$ is monotonic, i.e., $\xi'(X) \le \xi'(Y)$ if
$X\subseteq Y$ is an essential subsurface.  (From now on we implicitly
understand subsurfaces to be essential, and defined up to isotopy.)

If $W$ is a subsurface and $\gamma$ a curve in $S$ we say that
{\em $W$ and $\gamma$ overlap}, or $W\pitchfork \gamma$,  if $\gamma$
cannot be isotoped outside of $W$. We say that two surfaces $W$ and
$V$ overlap, or $W\pitchfork V$, if neither can be isotoped into the
other or into its complement. Equivalently, $W\pitchfork V$ iff
$W\pitchfork \boundary V$ and $V\pitchfork \boundary W$. 

See \cite{BKMM} for a careful discussion of these and related
notions.

\subsubsection*{Curves and markings}
The {\em curve complex}, $\CC(S)$, is a complex whose vertices are
essential simple closed curves up to homotopy, and whose $k$-simplices
correspond to $(k+1)$-tuples of disjoint curves.
Endow the 1-skeleton $\CC_1(S)$ with a path metric
giving each edge length 1. With this metric $\CC_1(S)$ is a $\delta$--hyperbolic
metric space \cite{MasurMinsky:complex1}. 

The definition of $\CC(W)$ is slightly different for $\xi(W)\le 1$:
If $W$ is a torus with at most one puncture then edges correspond to
pairs of curves intersecting once, and if $W$ is a 
sphere with 4 punctures then edges correspond to pairs of vertices
intersecting twice. In all these cases $\CC(W)$ is
isomorphic to the Farey graph. If $W$ is an annulus embedded in a
larger surface $S$ then $W$ admits a natural compactification as an
annulus with boundary and $\CC(W)$ is the set of homotopy classes of
essential arcs in $W$ rel endpoints, with edges corresponding to arcs
with disjoint interior. In this case $\CC(W)$ admits a quasi-isometry
to $\Z$ which takes Dehn twists to translation by 1. 
See \cite{MasurMinsky:complex2} for details. 

The {\em marking graph} $\MM(S)$ is a locally finite, connected graph whose vertices are
{\em complete markings} on $S$ and whose edges are {\em elementary moves.}
A complete marking is a system of closed curves consisting of a {\em
  base}, which is a maximal simplex in $\CC(S)$, together with a
choice of transversal curve for each element of the base, satisfying
certain intersection properties. For a detailed discussion, and proofs
of the properties we will list below, see \cite{MasurMinsky:complex2} or 
\cite{BKMM}. 

We make $\MM(S)$ into a path metric space by again assigning length 1 to
edges. We will denote distance in $\MM(S)$ as $d_{\MM(S)}(\mu,\nu)$, or
sometimes just $d(\mu,\nu)$. 
We will need to use the fact that  $\MCG(S)$ acts on $\MM(S)$, and any orbit map $g \mapsto
  g(\mu_0)$ induces a quasi-isometry from $\MCG(S)$ to $\MM(S)$. 

For an annulus $W\subset S$, we identify $\MM(W)$ with $\Z$, and map
this to $\CC(W)$ via the twist-equivariant quasi-isometry mentioned above. 

%% \item For each essential subsurface $W\subseteq S$ there is a
%%   coarse-Lipschitz map $\pi_W\co\MM(S)\to \CC(W)$, with coarseness 
%%   constants uniform over all $W$. This map 
%%   is also defined
%%   for vertices of $\CC(S)$ which have essential intersection with
%%   $W$. It is induced by surgery on intersections with $W$, and
%%   is the identity for curves already in $W$. 
%%   For $\mu,\nu\in\MM(S)$, we define
%%   $d_W(\mu,\nu)$ to be $d_{\CC_1(W)}(\pi_W(\mu),\pi_W(\nu))$. 
%% \item There is also a map $\pi_{\MM(U)}\co \MM(S)\to\MM(U)$, uniformly
%%   Lipschitz, which respects the $\CC(W)$--projections. That is, if
%%   $W\subset U$, then $\pi_W\circ\pi_{\MM(U)}$ and $\pi_W$ differ by a
%%   uniformly bounded amount. 
%% \end{enumerate}

%% We will only consider distance between vertices of $\MM(S)$ and we 
%% define the distance between a pair of vertices to 
%% be the minimal number of edges in any path between them. Having the 
%% metric on $\MM(S)$ be discrete will be slightly advantageous to us, as 
%% then the centroid map constructed in 
%% Section~\ref{centroid construction} will be Lipschitz instead of 
%% coarsely Lipschitz (as is the case in the analogous construction in a 
%% $\delta$--hyperbolic space).

\subsubsection*{Projections}
Given a curve in $S$ that intersects essentially a subsurface $W$, we
can apply a surgery to the intersection to obtain a curve in
$W$. This gives a partially-defined map from $\CC(S)$ to $\CC(W)$
which we call a {\em subsurface projection}, and in fact this
construction extends to a system of maps of both curve and marking complexes
that fit into coarsely commutative diagrams:
\begin{equation}
\xymatrix{
\MM(S) \ar[r]^{\pi_S}\ar[d]^{\pi_{\MM(W)}}  & \CC(S)\ar[d]^{\pi_W}\\
\MM(W) \ar[r]^{\pi_W}& \CC(W)\\
}
\end{equation}
Here $W$ is an essential subsurface of $S$, and we follow the
convention of denoting a map $\pi_W$ if its target is $\CC(W)$ and
$\pi_{\MM(W)}$ if its target is $\MM(W)$. The vertical $\pi_W$ is only
partially defined, namely just for those curves in $S$ that intersect
$W$ essentially. The horizontal maps take each marking to an
(arbitrary) vertex of its base (except in the case an an annulus).
By ``coarsely commutative'' we mean that the diagram
commutes up to errors bounded by a constant depending only on the
topological type of $S$.  We will also use the fact that
$\pi_{\MM(W)}$ is coarse-Lipschitz, with constants depending on the
topological type of $S$. Similarly, $\pi_W$ is coarse Lipschitz in a
restricted sense: if $a,b\in\CC(S)$ both intersect $W$ and $d_S(a,b) \le
1$ then $d_W(\pi_W(a),\pi_W(b)) \le 3$.  (Again, see \cite{MasurMinsky:complex2} and
\cite{BKMM} for details.)

\subsubsection*{Quasidistance formula}
In \cite{MasurMinsky:complex2} an approximation formula for
distances in $\MM(S)$ is obtained.  To state this,  define
the \emph{threshold function} $\truncate{x}{A}$ to be $x$ if $x\ge A$ and
$0$ otherwise. We define $x\approx y$ to mean $x\le ay+b$ and $y\le
ax+b$, where in the sequel 
$a$ and $b$ will typically be constants depending only on the
topological type of $S$ or on previously chosen constants.

For $\mu,\mu'\in \MM(S)$ and $W\subseteq S$, we define the abbreviation
$$
d_W(\mu,\mu') = d_{\CC_1(W)}(\pi_W(\mu),\pi_W(\mu')).
$$
We then have: 

\begin{theorem}{distance formula}{\rm (Quasidistance formula).}
There exists a constant $A_0 \ge 0$ depending only on the topology of
$S$ such that for each $A \ge A_0$, and for any $\mu,\mu' \in \MM(S)$ 
we have
$$d_{\MM(S)}(\mu,\mu')
\approx \sum_{Y\subseteq S} \truncate{d_{Y}(\mu,\mu')}{A}
$$
and the constants of approximation depend only on $A$ and on the 
topological type of~$S$. 
\end{theorem}

As a corollary of this, we note:
\begin{corollary}{C bounds M}
For any $r$ there exists $t$ such that for any $\mu,\nu\in\MM(S)$, if
$d_W(\mu,\nu) \le r$ for all $W\subseteq S$, then 
$d_{\MM(S)}(\mu,\nu) \le t$. 
\end{corollary}

\subsubsection*{Product regions}

If $W = W_1\union\cdots\union W_k$ is a disconnected surface with components $W_i$
we define $\MM(W)$ to be
$$
\MM(W) = \MM(W_1)\times\cdots\times\MM(W_k)
$$
metrized with the $\ell^1$ sum of the metrics on the factors. 
As a matter of convention we allow $\MM(W)$ to refer to a disconnected
surface, but only consider $\CC(W)$  when $W$ is connected. Such
products occur when considering certain regions in $\MM(S)$:

If $\Delta$ is a curve system in $S$, let $\QQ(\Delta)$ denote the set
of markings containing $\Delta$ in their base. This set admits a natural product
structure, described in \cite{BehrstockMinsky:rankconj} and
\cite{BKMM}: Let $\sigma(\Delta)$ denote the set of components of
$S\setminus \Delta$ that are not 3-holed spheres, together with the
annuli whose cores are components of $\Delta$. 
The following lemma is a consequence of the quasidistance formula: 
\begin{lemma}{product regions}
Given a curve system $\Delta$, there is a quasi-isometry
$$
\QQ(\Delta) \to \prod_{W\in\sigma(\Delta)} \MM(W)
$$
with constants depending only on the topological type of $S$, 
which is given by the product of projection maps, 
$\prod_{W\in\sigma(\Delta)} \pi_{\MM(W)}$.
\end{lemma}
As a special case, 
if $U\subset S$ is a (possibly
disconnected) surface, let $U^c$ be the surface consisting of all
components of $\sigma(\boundary U)$ which are not components of $U$.
Note that $\xi'(U)+\xi'(U^c)=\xi'(S)=\xi(S)$. Lemma \ref{product
  regions} gives a quasi-isometry
$$
\QQ(\boundary U) \to \MM(U) \times \MM(U^c).
$$

\subsubsection*{Projection bounds}
The projections $\pi_W$ satisfy a number of useful inequalities. 
One,  from \cite{Behrstock:asymptotic}, is: 
\begin{lemma}{behrstock inequality}
There exists a universal constant $m_0$ such that
for any marking $\mu\in\MM(S)$ and subsurfaces $V\pitchfork W$, 
$$
\min \left(  d_W(\mu,\boundary V) , d_V(\mu,\boundary W) \right) < m_0.
$$
\end{lemma}

The {\em geodesic projection lemma} \cite{MasurMinsky:complex2} states:
\begin{lemma}{geodesic projection}
Let $Y$ be a connected essential subsurface of $S$ satisfying $\xi(Y)\ne 3$
and let $g$ be a geodesic segment in $\CC(S)$ for which
$Y\pitchfork v$ for every vertex $v$ of $g$.
Then 
$$\diam_Y(g) \le B,$$
where $B$ is a constant depending only on $\xi(S)$.
\end{lemma}

The following generalization of the geodesic projection lemma is proven in
\cite[Lemma 5.5]{BKMM}:
\begin{lemma}{gen geodesic projection}
    Let $V, W\subseteq S$ be essential subsurfaces such that 
    $W \pitchfork \boundary V$. Let $g$ be a geodesic in $\CC(W)$. If
    $$d_W(g,\boundary V) > m_1$$ 
    then 
    $$\diam_V(g) \le m_2.$$
    Where the constants $m_1,m_2$ depend only on $S$. 
\end{lemma}

\subsubsection*{Partial orders}
The inequalities of Lemmas \ref{behrstock inequality} and 
\ref{geodesic projection} can be interpreted as describing a family of
partial orders for connected subsurfaces that are ``between'' pairs of markings
in $\MM(S)$. Given $x,y\in\MM(S)$, and a constant $c>0$, define for a
natural number $k$
$$
\FF_k(x,y) = \{U\subsetneq S: d_U(x,y)> kc\}.
$$
Note that this is (for appropriate threshold) the set of proper
(connected) subsurfaces
participating in the quasidistance formula for $d(x,y)$; in particular
it is finite. 
Define also a family of relations $\prec_k$ on proper connected subsurfaces of $S$, by
saying that $V\prec_k W$ if and only if $V\pitchfork W$, and 
$$
d_V(x,\boundary W) > kc.
$$
Note that $\prec_k$ depends on $x$, not $y$, so we assume throughout an
{\em ordered} pair $(x,y)$. 
In \cite{BKMM} we show that
\begin{lemma}{partial order}
There exists $c_0$ such that, if $c>c_0$ in the above definitions,
then for $k>2$ the relation $\prec_{k-1}$ is a partial order on
$\FF_k(x,y)$ for any $x, y\in\MM(S)$. Moreover if $V,W\in\FF_k(x,y)$
and $V\pitchfork W$ then $V$ and $W$ are $\prec_{k-1}$--ordered. 

\end{lemma}
Moreover, it will be useful to see that the relation $\prec_{k-1}$ can
be characterized in a few ways: 
\begin{lemma}{four inequalities}
Let $V,W\in \FF_k(x,y)$, $W\pitchfork V$, and take $c$ to be any 
sufficently large real number. 
The following are equivalent:
\begin{enumerate}
    \item $W\prec_{k-1}V$
\item $d_W(x,\boundary V) > (k-1)c$,
\item $d_W(y,\boundary V) \le c$,
\item $d_V(x,\boundary W) \le c$,
\item $d_V(y,\boundary W) > (k-1)c$.
\end{enumerate}
A related fact is that if 
$V\pitchfork \boundary W$, $d_V(x,\boundary W)>(k+1)c$ 
and $W\in\FF_2(x,y)$, then $V\in \FF_{k}(x,y)$.
\end{lemma}

These partial orders are closely related to the ``time-order''
that appears in \cite{MasurMinsky:complex2}. In \cite{BKMM}
these facts are established using just the projection inequalities, as an
extended exercise in  the triangle inequality.

\subsubsection*{Consistency Theorem}
Consider the combined projection map
$$
\Pi:\MM(S) \to \prod_{W\subseteq S} \CC(W),
$$
$\Pi(\mu) = (\pi_W(\mu))_W$, 
where $W$ varies over essential subsurfaces of $S$ and $\pi_W$ denotes
the subsurface projection map $\MM(S) \to \CC(W)$. We say that an
element $x = (x_W) \in \prod_W \CC(W)$ is {\em $D$--close to the image
  of $\Pi$} if there exists $\mu\in\MM(S)$ such that
$d_W(x_W,\mu) < D$ for all $W\subseteq S$. The following 
\emph{Consistency Theorem}, from
\cite{BKMM}, gives a coarse characterization of the image of $\Pi$. 

\begin{theorem}{consistency}{\rm (Consistency Theorem).}
Given $c_1,c_2 > 0$ there exists
$D$ such that any point $(x_W)_W \in
\prod_{W}\CC(W)$ satisfying the following two conditions 
is $D$--close to the image of $\Pi$.
\begin{itemize}
\item[C1:] For any $U\subset V \subseteq S$, if $d_V(\boundary
  U, x_V) \ge c_1$ then 
$$
d_U(x_U,x_V) < c_2.
$$
\item[C2:] For any $U,V\subset S$ with $U \pitchfork V$, 
$$
\min \left(    d_U(x_U,\boundary V) , d_V(x_V,\boundary
U) \right) < c_2.
$$
\end{itemize} 

Conversely given $D$ there exist $c_1,c_2$ so that if $(x_W)$ is
$D$--close to the image of $\Pi$ then it satisfies conditions C1-2. 
\end{theorem}

Note that the converse direction of the theorem includes Lemma
\ref{behrstock inequality} as condition C2 in the case $D=0$. 

\subsubsection*{$\Sigma$--hulls}
If $x,y\in\MM(S)$ and $W$ is a connected subsurface, let $[x,y]_W$ be
a geodesic in $\CC(S)$ connecting $\pi_W(x)$ to $\pi_W(y)$ (this may
not be unique but we can make an arbitrary choice --- all of them are
$\delta$--close to each other by hyperbolicity). For a finite set
$A\subset \MM(S)$, define $\hull_W(A)$ to be the union of $[a,b]_W$
over $a,b\in A$. For any fixed $\ep>0$, we define a $\Sigma$--hull of 
$A$ to be a set of the following form:
$$
\Sigma_\ep(A) = \{ \mu\in\MM(S): \forall W\subseteq S, \ 
     d_W(\mu,\hull_W(A)) \le \ep\}.
$$
If $A$ is a pair $\{x,y\}$ we also write $\Sigma_\ep(A) =
\Sigma_\ep(x,y)$. In \cite{BKMM} we study these sets, and in particular prove the
following: 
\begin{lemma}{Sigma properties}
Given $\ep$ and $n$ there exists $b$ such that
$$\diam(\Sigma_\ep(A)) \le b(\diam(A) + 1)$$
for any $A\subset \MM(S)$ of cardinality $n$. 

\end{lemma}

\subsubsection*{Tight geodesics, footprints and hierarchies}
A geodesic in $\CC(S)$ is a sequence of vertices $\{v_i\}$ such that
$d(v_i,v_j) = |j-i|$. In \cite{MasurMinsky:complex2} this is generalized
a bit to sequences of {\em simplices}, i.e., disjoint curve
systems $\{w_i\}$, such that $d(v_i,v_j) = |j-i|$ for any $v_i\in
w_i$, $v_j\in w_j$, and $i\ne j$. For a (generalized) geodesic
$g=\{w_i\}$ in
$\CC(V)$, and any subsurface $U\subset V$, we define the {\em
  footprint} $\phi_g(U)$ to be the set of simplices $w_i$ disjoint
from $U$. By the triangle inequality $\diam_V(\phi_g(U))\le 2$. A
condition called {\em tightness} is formulated in
\cite{MasurMinsky:complex2} which has the following property: for a
tight geodesic, all nonempty footprints are {\em contiguous} intervals
of one, two, or three simplices (leaving out the possibility of two
simplices at distance 2, with their midpoint not included). This is
the basic definition that leads to the notion of a {\em hierarchy of
tight geodesics} between any two $x,y\in\MM(S)$. A hierarchy consists
of a particular collection of tight geodesics $k$, each in $\CC(W)$ 
for a subsurface $W\subseteq S$ known as the support of $k$.
We will only need a few basic facts about hierarchies, to be used in the proof
of Lemma \ref{G diameter sum}.

\begin{lemma}{hierarchy facts}
If $x,y\in\MM(S)$ and $H=H(x,y)$ is a hierarchy of tight geodesics,
then 
\begin{enumerate}
\item $H$ contains a tight geodesic $[x,y]_S$
with support $S$ and endpoints $\pi_S(x)$ and $\pi_S(y)$ (the
selected points in the base of $x$ and $y$ respectively).

\item If $h$ is a tight geodesic in $H$ and $U$ is its support, then
  the endpoints of $h$ are within uniform distance $m_3$ in $\CC(U)$ from
  $\pi_U(x)$ and $\pi_U(y)$. 
\item If $h$ is a tight geodesic in $H$ and $U\subsetneq S$ is its support, 
then there exists $k$ in $H$ with support $W$, and a simplex $w$ in
$\phi_k(U)$, such that $U$ is either a component of $W\setminus w$, or
an annulus whose core is a component of $w$. 
\item A subsurface $W$ can be the support of at most one geodesic in
  $H$, which we denote $h_{x,y,W}$.
\item For a uniform $m_4$, all connected subsurfaces $W$ with
  $d_W(x,y) > m_4$ are domains of geodesics in $H$. 
\end{enumerate}
\end{lemma}

\subsubsection*{Rapid Decay}
\medskip

Although in the text we do not work directly with the rapid 
decay property, for the benefit of the reader who (like the authors) 
is not an analyst, we briefly discuss the formulation of the 
rapid decay property and related notions. For further details 
see \cite{Valette:BCbook} or \cite{Connes:noncommutative}. 

Given a finitely generated group $G$, we consider its action by
left-translation on $l^2(G)$, the square-summable $\C$-valued
functions on $G$. This action extends by linearity to an action of the group
algebra $\C G$, and indeed $\C G$ is just the subset of $l^2(G)$
consisting of functions with finite support, and the action is nothing
more than convolution, i.e., $f * g (z) = \sum_{x\in G}
f(x)g(x^{-1}z)$. 

This gives us an embedding of $\C G$ into the bounded operators on
$l^2(G)$, indeed  for $f\in \C G$ and $h\in l^2(G)$ we have $||f*h||_2
\le ||f||_1 ||h||_2$ by Young's inequality.
The reduced $C^*$-algebra of $G$, denoted $C^*_r(G)$, is 
the closure of $\C G$ in the operator norm. 

On the other hand, $\C G$ embeds in the normed spaces 
$H^s(G) = \{h: ||h||_{2,s}< \infty\}$, where 
$$||h||_{2,s} = (\sum_{x\in G}
((1+|x|)^sh(x))^2)^{1/2}$$
for $s>0$, and $|\cdot|$ denotes word
length in $G$.  The intersection
$H^\infty(G)=\intersect_s H^s(G)$ is the space of {\em rapidly decreasing
  functions} on $G$. (Note that $H^s$ and $H^\infty$ are invariant, up
to bounded change of norm, under change of generators).

Recall that, in the abelian setting (e.g., $G=\Z^n$), functions of rapid decrease in
$G$  Fourier-transform to smooth functions on the
Pontryagin dual  $\hhat G$ (e.g., $\hhat{\Z^n}=T^n$). In the nonabelian
setting, there is no 
Pontryagin dual so $H^\infty$ acts as a substitute for the algebra of
smooth functions (see Connes-Moscovici
\cite{ConnesMoscovici:hyperbolic}). 

We say that $G$ has the {\em rapid decay property} if the embedding of $\C
G$ into $C^*_r(G)$ extends continuously to an embedding of
$H^\infty(G)$. 

This condition boils down
(see \cite{ChatterjiRuane,DrutuSapir:RD}) to a polynomial convolution norm bound of the following
form: there exists a polynomial $P(s)$ such that, if $f$ is
supported in a ball of radius $s$ in $G$, then
\begin{equation}\label{convolution bound}
||f*g||_2 \le P(s)||f||_2||g||_2.
\end{equation}
Here one can start to see at least the relevance of the centroid
condition and its bound. Indeed, in the sum $f*g(z)=
\sum_{x}f(x)g(x^{-1}z)$,
$x$ can be restricted to the ball of
radius $s$. Now we can rearrange this as a sum over the centroids
$t = \kappa(1,x,z)$, and the number of such $t$ is polynomial in
$s$ by Theorem \ref{centroid}. This observation plays a role in the
proofs of (\ref{convolution bound}) 
in both Dru\c{t}u-Sapir
\cite{DrutuSapir:RD} and Chatterji-Ruane
\cite{ChatterjiRuane}.

\section{Centroids}
\label{centroid construction}

In a $\delta$--hyperbolic metric space $X$, define a {\em
  $\rho$--centroid} of a triple of points $A=\{a_1,a_2,a_3\}$ to be a
point $x$ which is within $\rho$ of each of the geodesics $[a_i,a_j]$ 
(if geodesics are not unique make an arbitrary choice). Hyperbolicity
implies that $\rho$--centroids always exist for a uniform $\rho$
(depending on $\delta$), and indeed this condition is equivalent to
hyperbolicity. The next lemma states a few more facts that we need; the
proof, which is an exercise, is left out. 
\begin{lemma}{hyperbolic centroids}
Let $X$ be a $\delta$--hyperbolic geodesic metric space. There exist
$\delta_0,L>0$ and a function $D$, depending only on $\delta$, such that
\begin{enumerate}
\item Every triple has a $\rho$--centroid if $\rho\ge \delta_0$.
\item The diameter of the set of $\rho$--centroids of any triple is at
  most $D(\rho)$.
\item The map taking a triple to the set of its $\rho$--centroids is
  $L$--coarse-Lipschitz in the Hausdorff metric.
\end{enumerate}
\end{lemma}

In this section we will utilize this idea to 
give a centroid map for $\MCG(S)$ (or equivalently $\MM(S)$)
satisfying the first three properties of Theorem \ref{centroid}.

\begin{theorem}{centroid from projections}
There exists $\ep,\rho>0$ and  a map 
$\kappa\co\MM(S)^3 \to \MM(S)$ with the following
properties: 
\begin{enumerate}
\item $\kappa(a,b,c)$ is invariant under permutation of the arguments. 
\item $\kappa(ga,gb,gc)=g\kappa(a,b,c)$ for any $g\in\MCG(S)$.
\item $\kappa$ is Lipschitz. 
\item $\kappa(a,b,c) \in \Sigma_\ep(a,b,c)$. 
\item For each $W\subseteq S$, $\pi_W(\kappa(a,b,c))$ is a
  $\rho$--centroid of the triangle, in $\CC(W)$, with vertices $\pi_W(a)$,
$\pi_W(b)$, and $\pi_W(c)$.
\end{enumerate}
\end{theorem}

\begin{proof}
Let $\delta$ be a hyperbolicity constant for $\CC(W)$ for all
$W\subseteq S$, let $\delta_0$ be the constant given in Lemma
\ref{hyperbolic centroids}, and fix $\ep>\delta_0$. 
For each $\MCG$--orbit of triple $(a,b,c)$ choose a representative 
$A=\{a_1,a_2,a_3\}\subset\MM(S)$. For any $W\subseteq S$,  
the triangle $\union_{i,j} [\pi_W(a_i),\pi_W(a_j)]$ (which is 
coarsely equal to $\hull_W(A)$) has an $\ep$--centroid. 
Choose such an $\ep$--centroid and call it $x_W$. 
We will show
that $(x_W)\in\prod_W \CC(W)$ satisfies conditions C1-2 of the
consistency theorem, for suitable $c_1,c_2$.

Consider $V,W\subset S$ satisfying $\boundary V \pitchfork W$. 
Suppose that $d_W(x_W, \boundary V) > D(\max(m_0,\ep))$, where 
$D(\cdot)$ is the function given by Lemma~\ref{hyperbolic centroids}.
Then by part (2) of Lemma \ref{hyperbolic centroids}, $\pi_W(\boundary V)$ is
not an $m_0$--centroid for $\pi_W(A)$, and so  for at least one leg $g$
of $\hull_W(A)$, $d_W(\boundary V,g) > m_0$. 
Suppose without loss
of generality that $g =  [\pi_W(a_1),\pi_W(a_2)]$. 

By Lemma \ref{gen geodesic projection}, this gives us an upper bound
$\diam_V(g) \le m_2$. In other words $\pi_V(a_1)$ and $\pi_V(a_2)$ are
close together and hence the centroid $x_V$ is close to both. 

Since $x_W$ is within $\ep$ of $g$ in $\CC(W)$, and $d_W(\boundary
V,g) > \ep$, we may connect $x_W$ to $g$ by a path of length at most 
$\ep$ consisting of
curves that all intersect $V$, and so the Lipschitz property of
$\pi_V$ gives an upper bound on $d_V(x_W,g)$. From this and the 
previous paragraph, we conclude that there
is a bound of the form
$$
d_V(x_V,x_W) < m_3
$$
for suitable uniform $m_3$. 

When $V\subset W$, this establishes C1. 

When $V\not\subset W$, i.e., $V\pitchfork W$, 
since we have assumed $d_W(x_W,\boundary V)$ is large, 
the direction of the Consistency Theorem given by 
Lemma~\ref{behrstock inequality} yields a bound on 
$d_V(x_W,\boundary W)$. Since we already have a bound on 
$d_V(x_V,x_W)$ by the above, this in turn bounds 
$d_V(x_V,\boundary W)$, and thus establishes C2. 

Having established C1 and C2, we can apply the Consistency Theorem to
conclude that there exists $\mu\in\MM(S)$ with $d_W(\mu,x_W)$
uniformly bounded for
all $W$. Let this $\mu$ be $\kappa(A)$. 
(Note that, by the quasidistance formula, $\mu$ is determined
up to bounded error.) For any triple $A'=(a,b,c)$ satisfying 
$gA=A'$ define $\kappa(A')=g\kappa(A)$.

By construction, $\kappa$ satisfies conditions (2) and (5) of the theorem. 

Condition (1) is evident since the construction depended on the
unordered set $A$. Condition (3), the Lipschitz property, follows
immediately from the quasidistance formula and the 
coarse-Lipschitz property
of hyperbolic centroids (part (3) of Lemma \ref{hyperbolic 
centroids}); note that in this case we obtain a Lipschitz map and not 
just a coarse-Lipschitz one, since there is a lower bound on the 
distance in $\MM(S)$ between pairs of distinct points.
Condition (4) follows from (5) and the definition of $\Sigma_\ep$.
%This concludes the proof of
%Theorem~\ref{centroid from projections}.
\end{proof}

\section{Polynomial bounds}
\label{polynomial bounds}

It remains to prove that the map $\kappa$ of Theorem \ref{centroid from
  projections} satisfies the polynomial bound of part (4) of Theorem \ref{centroid}.
That is, letting
$$
K(x,y,r) = \left\{\strut \kappa(x,y,z): z\in \NN_r(x)\right\}
$$
(where $\NN_r(Y)$ denotes a neighborhood of $Y\subset\MM(S)$ of 
radius $r$)
we need to
find a polynomial in $r$, independent of $x$ and $y$, which bounds
$\#K(x,y,r)$. Throughout this section, $r$ will denote a 
fixed positive real 
number and the points $x,z$ will satisfy $z\in \NN_r(x)$.

\subsection{Reduction to $\Sigma$--hulls}
We first reduce the problem to that of counting the number of elements
in a suitable $\Sigma$--hull. 
 
If $\mu = \kappa(x,y,z)$, then by definition of $\kappa$
and of hyperbolic centroids, 
for each $W\subseteq S$ the set $\pi_W(\mu)$ is in a uniformly
bounded neighborhood of $[x,y]_W$. It follows that
$$K(x,y,r)\subset\Sigma_{\ep'}(x,y)$$
for suitable $\ep'$ (depending on $\ep$ and the hyperbolicity
constant).
Moreover, $\mu\in\Sigma_{\ep'}(x,z)$ by the same argument, so 
$d(x,\mu) \le \diam(\Sigma_{\ep'}(x,z)) \le b(d(x,z)+1)$ (by Lemma
\ref{Sigma properties}) and thus $\mu\in \NN_{b(r+1)}(x)$. 
Hence (changing variable names) it suffices to give a polynomial bound
in $r$, depending on $\ep$ but not on $(x,y)$,  on the cardinality of 
$$A(x,y,r)=\Sigma_{\ep}(x,y)\cap \NN_{r}(x).$$

The following lemma indicates that $A(x,y,r)$ is contained in a 
$\Sigma$--hull that gives a good estimate on its size:
\begin{lemma}{reduce to Sigma}
Given $\ep$ there exists $\ep'$ such that, 
 for each $x,y\in \MM(S)$ and each $r>1$, there exists
    $q\in\Sigma_{\ep'}(x,y)$ for which 
$$
\Sigma_\ep(x,y)\cap \NN_{r}(x) \subseteq
    \Sigma_{\ep'}(x,q)
$$ 
and satisfying $d(x,q)<a r$, with $b$ depending only
    on $\epsilon$ and $\xi(S)$.
\end{lemma}

\begin{proof} 
For each $W\subseteq S$, by the 
definition of $\Sigma$--hulls, we have that $\pi_{W}(A(x,y,r))$ is
contained in the $\epsilon$--neighborhood of a $\CC(W)$--geodesic
$[x,y]_W$
between a point of $\pi_{W}(x)$ and a point of $\pi_{W}(y)$. 
Let $m_W$ denote the
vertex of $[x,y]_W$ which is within $\ep$ of $\pi_W(A(x,y,r))$ and 
is farthest from $\pi_{W}(x)$.

\subsubsection*{Finding $q$ using consistency}
We now claim that the tuple, $(m_W)_{W\subseteq S}$, satisfies the consistency
conditions of Theorem~\ref{consistency} 
and thus gives rise to a marking, which we will call~$q$.

For convenience we note that we can simultaneously establish 
both C1 and C2 by showing the following: there exists a uniform 
constant such that for any $U,V\subseteq S$ such
that $\boundary U\pitchfork V$, either $d_V(m_V,\boundary U)$ or 
$d_U(m_U,\boundary V \union m_V)$ is bounded by this constant. 
Here by $\boundary V \union m_V$ we mean the union as curve systems,
or equivalently the join as simplices in $\CC(S)$.

Thus, let $U$ and $V$ be such that $\boundary
U\pitchfork V$ and 
\begin{equation}\label{bdry U mV far}
d_V(m_V,\boundary U) > 2(\ep+2).
\end{equation}

Let $\mu\in A(x,y,r)$ be such that $\pi_V(\mu)$ is within $\ep$ of
$m_V$, and let $\nu\in A(x,y,r)$ be such that $\pi_U(\nu)$ is within
$\ep$ of $m_U$.  Since $d_V(\boundary U,m_V) > 2\ep+2$, there is a
$\CC(V)$--path from $\pi_V(\mu)$ to $m_V$  consisting of curves that
intersect $U$, so by the Lipschitz property of $\pi_U$ we have a bound
$$
d_U(m_V,\mu) < b_1
$$
for some uniform $b_1$. In fact, since $m_V$ and $\boundary V$ are
disjoint we may instead write 
\begin{equation}\label{mu near mV in U}
d_U(m_V\union\boundary V,\mu) < b_{1}.
\end{equation}

Since $m_V$ lies on $[x,y]_V$, the bound (\ref{bdry U mV far}) also
implies, by the triangle inequality, that $\pi_V(\boundary U)$ cannot
be within $\ep+2$ of both $[x,m_V]_V$ and $[m_V,y]_V$. This gives us 
two cases which we treat separately.

\subsubsection*{Case a:} Suppose $\pi_V(\boundary U)$ is more than $\ep+2$
from $[x,m_V]_V$. 

Let $\sigma\in A(x,y,r)$, and 
let $t\in[x,y]_V$ be within $\ep$ of $\pi_V(\sigma)$. By definition of
$m_V$, $t$ must be in $[x,m_V]$. Hence, again by the Lipschitz
property of $\pi_U$, we have
$$
d_U(\sigma,t) \le b_1.
$$
The bounded geodesic projection lemma implies that
$$
d_U(x,t) \le B,
$$
so we have a bound on $d_U(x,\sigma)$. Applying this to $\nu$, 
which was chosen to satisfy $d_{U}(\nu,m_{U})< \ep$, we get
$$
d_U(x,m_U) \le b_2
$$
for suitable $b_2$. 
By construction of $m_{U}$, the above bound implies 
$d_U(x,\mu)\le b_{2}$ as well. Hence, by the triangle inequality, 
for suitable $b_{3}$ we have:  
$$
d_U(m_U,\mu) \le b_3.
$$ 
Hence by (\ref{mu near mV in U}) we conclude there exists a uniform 
constant, $b_{4}$, satisfying:
$$
d_U(m_U,m_V\union \boundary V) \le b_4
$$
which is what we wanted to show. 

\subsubsection*{Case b:}
Suppose $\pi_V(\boundary U)$ is more than $\ep+2$ from $[m_V,y]_V$.

Then, by the geodesic projection lemma, 
$$
d_U(m_V,y) \le B,
$$
and applying (\ref{mu near mV in U}) again we have
$$
d_U(\mu,y)\le b_5.
$$
Let $t\in[x,y]_U$ be within $\ep$ of $\pi_U(\mu)$ --- then we have
$$d_U(t,y) \le b_5+\ep.
$$
By definition, $m_U$ is  in $[t,y]_U$, so by
the triangle inequality
$$
d_U(m_U,m_V\union \boundary V) \le b_6, 
$$
and again we are done.

Having established C1-2 for $(m_W)$, the Consistency Theorem gives us
$q\in\MM(S)$ such that 
\begin{equation}
    \label{equation:b7}d_W(q,m_W) < b_7
\end{equation}
for a uniform $b_7$. 
By definition of $m_W$, we have that $\pi_W(A(x,y,r))$ is within $\ep$ of
$[x,m_W]_W$ for each $W$, and hence by hyperbolicity of $\CC(W)$ this 
set is within a suitable $\ep'$ of $[x,q]_W$. In other words,
$$
A(x,y,r) \subset \Sigma_{\ep'}(x,q).
$$

\subsubsection*{Bounding $d(x,q)$}
It remains to check that $d_{\MM(S)}(x,q)<br$, for a uniform $b$. 

Fix $c > \max\{A_{0}, m_{0}+b_{7}+\ep+ 3\}$, where $A_{0}, m_{0}$, and 
$b_{7}$ are the constants of 
Theorem~\ref{distance formula}, Lemma~\ref{behrstock inequality}, and 
Equation~(\ref{equation:b7}). Recall from Section~\ref{background} 
the set of subsurfaces
$\FF_{3}(x,q)$, with its partial ordering $\prec_2$, where
$U\in\FF_3(x,q)$ iff $d_U(x,q) > 3c$. Using $3c$ as a threshold in the
quasidistance formula, we have
$$
d(x,q) \approx \sum_{W\in\FF_3(x,q)} d_W(x,q), 
$$
which we will use to get an upper bound on $d(x,q)$. Let $\UU$ be the
set of maximal elements of $\FF_3(x,q)$ with respect to the partial
order $\prec_2$. Since 
overlapping subsurfaces are $\prec_2$--ordered (Lemma \ref{partial order}),
the elements $U_i$ of $\UU$ are either disjoint or nested, which means
there are at most $2\xi(S)$ of them (the bound of $2\xi(S)$ is an 
easy exercise, but all that matters is that there is some universal 
constant).
%add a hint here? the proof is easy. just start with a pants decomp
%then start removing curves to produce a sequence of nested 
%subsurfaces

Let $u_i\in A(x,y,r)$ be a marking such that $d_{U_i}(u_i,m_{U_i}) <
\ep$ (this exists by definition of $m_{U_i}$). Moreover, 
by definition of $q$ we know that 
$d_{U_i}(q,m_{U_i})$ is uniformly bounded by $b_{7}$, hence we have
\begin{equation}\label{q ui close in Ui}
d_{U_i}(q,u_i) < b_7 + \ep.
\end{equation}
We claim that for each $W\in \FF_3(x,q)$, there exists  at least one of 
the $u_i$ which satisfies 
\begin{equation}\label{q ui close in W}
d_W(q,u_i) < b
\end{equation}
for some uniform constant $b$. 
For $W$ an element of $\UU$  we have just
established  this. For any other $W$, we must have $W\prec_2 U_i$ for
some $U_i\in\UU$, so $W\pitchfork U_i$ and so we 
have (from Lemma \ref{four inequalities}) the inequalities 
\begin{equation}\label{W less Ui 1}
d_W(q,\boundary U_i) < c
\end{equation}
and
\begin{equation}\label{W less Ui 2}
d_{U_i}(x,\boundary W) < c.
\end{equation}
Now from (\ref{q ui close in Ui}) and the fact that $d_{U_i}(x,q) >
3c$, the triangle inequality yields 
$d_{U_i}(x,u_i) > 3c-3-b_7-\ep > 2c$, where the $3$ is being subtracted 
because $\diam_{U_{i}}(u_{i})\le 3$. Together with 
(\ref{W less Ui 2}) we then have
$$
d_{U_i}(\boundary W, u_i) > c.
$$
Hence by Lemma \ref{behrstock inequality}, 
$$
d_{W}(\boundary U_i, u_i) < c_2.
$$
Now with (\ref{W less Ui 1}), we get an inequality of the form
$$
d_W(q,u_i) < b
$$
for a uniform $b$. This establishes (\ref{q ui close in W}), which 
proves the claim. 

Now this means that 
$$
d_W(x,q) \le d_W(x,u_i) + b.
$$
In the quasidistance formula (Theorem \ref{distance formula}), we may choose any sufficiently large
threshold at a cost of changing the approximation constants. Hence we
may choose a threshold greater than $2b$, and then for every term
$d_W(x,q)$ larger than the threshold we have 
$$
\frac12 d_W(x,q) \le d_W(x,u_i).
$$
It then follows that every term in the quasidistance formula for $d(x,q)$
appears, with a multiplicative error,
in the quasidistance formula for one of the $d(x,u_i)$. 
This means 
$$
d(x,q) \le \sum_i a' d(x,u_i) + b' \le 2\xi(S)( a' r + b')
$$
where $a',b'$ come from Theorem \ref{distance formula}. A bound of the form
$d(x,q) < ar$ follows since $r>1$. 
\end{proof}

\subsection{Polynomial bound on $\Sigma$--hulls}
Now that our set $K(x,y,r)$ is known to be contained in a
$\Sigma$--hull of comparable diameter, it suffices to 
obtain an appropriate polynomial bound on the $\Sigma$--hull of two  
points. 

\begin{theorem}{polynomial bound}
Given $\ep>0$, there is a constant $c=c(\ep,S)$,
such that, for any $x,y\in\MM(S)$,  
$$
\#\Sigma_\ep(x,y) \le c d(x,y)^{\xi(S)}. 
$$
\end{theorem}
%% (Recall $\xi'(S)=\xi(S)$, and for disconnected 
%% subsurfaces $\xi'$ is defined as in
%% Section \ref{background}.)  
Note by Lemma \ref{Sigma properties} that
$d(x,y)$ is interchangeable, up to bounded factor, with $\diam
(\Sigma_\ep(x,y))$, and we will freely make use of that. 

\begin{proof}
The idea of the proof is to (coarsely) cover $\Sigma_\ep(x,y)$ by sets
for which the desired inequality holds by induction, and the sum of
whose diameters is bounded by the diameter of $\Sigma_\ep(x,y)$. 
%Below we will work with disconnected surfaces and hence 
%measure complexity with $\xi'$, as 
%defined in Section \ref{background}, instead of $\xi$; 
%recall that for a connected surface 
%$\xi'(S)=\xi(S)$.

Given $x$ and $y$,
we will construct a finite collection $\Gamma=\Gamma_{x,y}$ of finite 
subsets of
$\MM(S)$, and a map $\gamma\co \Sigma_\ep(x,y) \to \Gamma$. The proof will
follow from three lemmas about this construction. 
The first one states that the bound of the theorem holds for the sets
in $\Gamma$:
\begin{lemma}{inductive bound}{\rm (Inductive Bound Lemma).}
For each $\GG\in\Gamma_{x,y}$,
\begin{equation*}
\#\GG \le b_1 \  \diam(\GG)^{\xi(S)}.
\end{equation*}
where $b_1$ depends only on $S$.
\end{lemma}

The second lemma implies that the collection $\Gamma$
coarsely covers $\Sigma_\ep(x,y)$:

\begin{lemma}{mu near G}{\rm (Covering Lemma).}
For each $\mu\in\Sigma_\ep(x,y)$, 
\begin{equation*}
\mu\in\NN_{b_2}(\gamma(\mu)).
\end{equation*}
where $b_2$ depends only on $S$.  
\end{lemma}

Finally, we will bound the sum of the diameters of the sets in
$\Gamma$:

\begin{lemma}{G diameter sum}{\rm (Diameter sum bound).}
For a constant $b_3$ depending only on~$S$, 
\begin{equation}\label{sum of diameters}
\sum_{\GG\in\Gamma_{x,y}} \diam(\GG) \le b_3\  \diam(\Sigma_\ep(x,y))
\end{equation}
\end{lemma}

The proof of Theorem \ref{polynomial bound} is then an inductive
argument. First we note that the statement of the theorem can be made
not just for $S$ but for any connected subsurface $W$ of $S$, and that
in order to correctly account for annuli we should write the
inequality as
\begin{equation}\label{subsurface polynomial bound}
\#\Sigma_{\ep,W}(x,y) \le c (\diam(\Sigma_{\ep,W}(x,y)))^{\xi'(W)}
\end{equation}
where $x,y\in\MM(W)$ and $\Sigma_{\ep,W}$ denotes the $\Sigma$-hull
within $\MM(W)$. Recall that $\xi'(W)=\xi(W)$ for connected
non-annular $W$, as in Section \ref{background}. For an annulus $W$,
$\xi'(W)=1$,  $\MM(W)$ is $\Z$, and
$\Sigma_{\ep,W}(x,y)$ is just the interval between $x$ and $y$. So
this establishes the base
case of (\ref{subsurface polynomial bound}).

We will establish the three lemmas for complexity $\xi(S)$,
where Lemma \ref{inductive bound} in particular will rely on assuming
(\ref{subsurface polynomial bound}) for all smaller complexities. 
Once that is done, for uniform constants $c_i$ we have
\begin{align*}
\#\Sigma_\ep(x,y) & \le c_1 \sum_{\GG\in\Gamma}\#\GG \\
 &\le c_2\sum_{\GG\in\Gamma}\diam(\GG)^{\xi(S)} \\ 
 & \le c_2 \left(\sum_{\GG\in\Gamma}\diam(\GG)\right)^{\xi(S)}\\
& \le c_3 \diam(\Sigma_\ep(x,y))^{\xi(S)}.
\end{align*}
where the first line follows from
Lemma \ref{mu near G}, together with a bound on the cardinality of 
$b_2$--balls in $\MM(S)$;
the second line follows from 
Lemma \ref{inductive bound}; the third from arithmetic; and the last from
Lemma \ref{G diameter sum}.

This gives the proof of Theorem \ref{polynomial bound}, modulo the
construction of $\Gamma$ and proofs of the three lemmas.

\subsubsection*{Construction of the cover}
Let $U$ be a (possibly disconnected) nonempty essential subsurface,
with its components denoted $U_1,\ldots,U_k$.
Recall from Section \ref{background} that 
$\MM(U) = \prod_i \MM(U_i)$, and define
$$
\Sigma_{\ep,U}(x,y) = \prod_i
\Sigma_{\ep,U_i}(\pi_{\MM(U_i)}(x),\pi_{\MM(U_i)}(y)).
$$
Now recalling from Lemma \ref{product regions} that $\QQ(\boundary U)$
is identified quasi-isometrically with $\MM(U) \times\MM(U^c)$, 
define 
$\GG(U,x,y) \subset \QQ(\boundary U)$ to be the set:
$$
\GG(U,x,y) =  \Sigma_{\ep,U}(x,y)
\times \{\pi_{\MM (U^c)}(y)\}. 
$$

We also need a degenerate form of this: if $p$ is a simplex in
$\CC(S)$, Lemma \ref{product regions} tells us that $\QQ(p)$ is
quasi-isometrically identified with 
$\prod_{W\in\sigma(p)} \MM(W)$, where $\sigma(p)$ is the decomposition
associated to $p$. Let $y_p$ denote the point corresponding to the tuple
$(\pi_{\MM(W)}(y))_{W\in\sigma(p)}$, and let $y'_p$ be
  an additional point at distance $1$ from $y_p$. We define
  $$\GG(p,y) = \{y_p,y'_p\}.$$
The second point is included for the technical purpose of making
 $\diam(\GG(p,y))$ equal to 1 instead of 0. Note that $\GG(p,y)$ is
{\em not} the same as $\GG(U,x,y)$ even when $U$ is a union of annuli
with cores comprising $p$.

Now we will construct our particular collection $\Gamma$ of sets of
this type, together with the map $\gamma\co\Sigma_\ep(x,y) \to \Gamma$. 
Let $\mu\in\Sigma_\ep(x,y)$, and 
let $\FF=\FF_3(\mu,y)$ be as in \S\ref{background}, taking  
$c=\ep+B+\frac{m_{4}}{2}+c_{0}$ where $B$, $c_{0}$, and $m_{4}$ are 
the constants from Lemmas~\ref{geodesic 
projection}, \ref{partial order},  and~\ref{hierarchy facts}. 
Lemma \ref{partial order} says that the relation $\prec_2$ is a partial
order on $\FF$. 
Assuming $\FF\ne\emptyset$, among all $\prec_2$--minimal subsurfaces,
consider the set $U$ of those 
that are maximal with respect to inclusion. Then $U$ is a union of
disjoint essential subsurfaces. 

Fix a constant $a>2\ep$.
Suppose that $d_{\CC_1(S)}(\mu,\boundary U) \le a.$
Then we let $\gamma(\mu) = \GG(U,x,y)$. 

Suppose that $d_{\CC_1(S)}(\mu,\boundary U) > a$.
Let $q(\mu)$ be the nearest point to $\mu$ on the tight geodesic
$[x,y]_S = h_{x,y,S}$ (see \S\ref{background} and Lemma~\ref{hierarchy facts}.)
In particular $d_S(\mu,q) \le \ep$ since $\mu\in\Sigma_\ep(x,y)$. 
%%%%(Throughout this section all curve-complex geodesics will be assumed
Let $p(\mu)$ be a simplex along the segment $[q,\pi_S(y)]\subset [x,y]_S$ which is at distance 
$a/2$ from $q$. Let $\gamma(\mu) = \GG(p,y)$.

If $\FF = \emptyset$, define $q$ and  $p$ as above, unless $d(q,y) <
a/2$ in which case let $p$ be the last simplex of 
$h_{x,y,S}$ (i.e., a subset of
$y$). Again let $\gamma(\mu) = \GG(p,y)$.

We let $\Gamma$ be the set of all $\gamma(\mu)$ thus obtained. 

\subsubsection*{Inductive Bound Lemma}
We now prove Lemma \ref{inductive bound}, relating cardinality to
diameter for $\GG\in\Gamma$. In this proof we will assume by induction
that Theorem \ref{polynomial bound}, or rather its subsurface version
(\ref{subsurface polynomial bound}), 
holds for all connected proper subsurfaces of $S$. 

If $\GG = \GG(p,y)$ then $\#\GG=2$ and $\diam(\GG)=1$,
so we are done.  Now consider $\GG=\GG(U,x,y)$. 
In each component $U_i$ of $U$ we have
$$
\#\Sigma_{\ep,U_i}(x,y) \le C
\diam(\Sigma_{\ep,U_i}(x,y))^{\xi'(U_i)}
$$
by (\ref{subsurface polynomial bound}).
Now since $\GG(U,x,y)$ is a
product of such sets over the components of $U$ and $\xi'$ is additive,
the bound follows with exponent $\xi'(U)$. Since $\xi'(U) \le \xi'(S)=\xi(S)$, 
Lemma \ref{inductive bound} follows. 

\subsubsection*{Covering Lemma}
We next prove Lemma \ref{mu near G}, which says that 
each $\mu\in\Sigma_\ep(x,y)$ is uniformly close to 
$\gamma(\mu)$. 

We estimate $d(\mu,\GG(U,x,y))$  via the following lemma:
\begin{lemma}{distance to G}
Let $\mu\in\Sigma_\ep(x,y)$ and $U$ a (possibly disconnected) subsurface. 
$$
d(\mu,\GG(U,x,y)) \approx \sum_{W\subseteq U^c} \truncate{d_W(\mu,y)}{L}
+ \sum_{W\pitchfork\boundary U} \truncate{d_W(\mu,\boundary U)}{L}
$$
for a uniform choice of constants.
\end{lemma}

\begin{proof}
For any $\mu\in\MM(S)$, define $\tau(\mu) \in \GG(U,x,y)$ as
follows. In view of Lemma \ref{product regions},  to describe
$\tau\in\QQ(\boundary U)$ we must simply give 
its restrictions $\pi_{\MM(V)}(\tau)$ to each component $V$ of $U$ and of
$U^c$. 

Hence, for each component $U_i$ of $U$, let $\pi_{\MM(U_i)}(\tau)
\equiv \pi_{\MM(U_i)}(\mu)$. For each component $V$ of $U^c$, let
$\pi_{\MM(V)}(\tau) \equiv \pi_{\MM(V)}(y)$. 

If $\mu\in\Sigma_\ep(x,y)$ then $\pi_{W}(\mu)$ is within $\ep$ of
$[x,y]_{W}$ for each $W$,  and hence for $W\subset U_i$
the same is true (perhaps with a change of
$\ep$) for $\pi_{W}(\tau)$, since $\pi_{\MM(U_i)}$ and $\pi_W$
(coarsely) commute. 

It follows that $\pi_{\MM(U_i)}(\tau)\in \Sigma_{\ep',U_i}(x,y)$. 
Thus, since $\ep'-\ep$ is uniformly bounded, the quasidistance formula 
yields 
that $\tau(\mu)$ is a uniformly bounded distance from $\GG(U,x,y)$; 
for the coarse measurements we make below it is no loss of 
generality to assume $\tau(\mu)\in\GG(U,x,y)$. 

Now the quasidistance formula gives
$$
d(\mu,\tau(\mu)) \approx \sum_W \truncate{d_W(\mu,\tau)}L
$$
but we notice that, for all $W\subset U_i$, the corresponding terms
are uniformly bounded. Thus by choosing $L$ sufficiently large those
terms disappear (at the expense of changing the constants implicit in
the ``$\approx$''). 

If $W\subseteq U^c$ then $d_W(\mu,\tau)$ is estimated by
$d_W(\mu,y)$ up to bounded error, so again possibly choosing $L$
larger we can replace one by the other at a bounded cost in the
constants. Finally, if $W\pitchfork \boundary U$ then, since $\tau$
contains $\boundary U$, we can replace those terms by
$d_W(\mu,\boundary U)$. (See \cite{BKMM} for other examples of this
type of argument). 

This gives an upper bound for $d(\mu,\GG(U,x,y))$ of exactly the type
in the lemma. To get the lower bound we observe that for {\em any} point
in $\GG(U,x,y)$ the terms of the given type must appear in its
quasidistance formula. 
\end{proof}

Consider the case that $\gamma(\mu) = \GG(U,x,y)$. 
To bound $d(\mu,\gamma(\mu))$ we must control both types of terms
that appear in Lemma~\ref{distance to G}.

In the case $W\pitchfork \boundary U$, notice that if 
$d_W(\mu,\boundary U)$ is sufficiently large, then by 
Lemma~\ref{four inequalities} it follows that 
$W\in \FF=\FF_3(\mu,y)$ and
$W$ precedes $U$ in the $\prec_2$--order, which 
contradicts the minimality of $U$. 
For the second, we see that if $d_W(\mu,y)$ is sufficiently large then
again $W$ would belong to $\FF$, and since it is disjoint from $U$, there
would be a $\prec_2$--minimal element disjoint from $U$, which again
contradicts the choice of $U$. 

This uniformly bounds all the terms in Lemma \ref{distance to G}, and
hence gives a uniform bound on $d(\mu,\gamma(\mu))$. (Again, this is
done by increasing the threshold past the uniform bound so that all
the terms disappear; all that is left is the additive error in
``$\approx$''.)

\medskip

Now consider the case that $\gamma(\mu) = \GG(p,y)$ with $p$ a
simplex in $[x,y]_S$. Recall this means that $\FF$ is either
empty, or its $\prec_2$--minimal elements are at $\CC(S)$--distance at least $a$ from
$\mu$. 

In fact  a bit more is true. If
$W\in\FF_3(\mu,y)$ then $d_W(x,y)>3c-\ep$, because
$\mu\in\Sigma_\ep(x,y)$ and hence $\pi_W(\mu)$ is $\ep$--close to 
$[x,y]_W$. Since $3c-\ep>2c>B$, the geodesic projection lemma (\ref{geodesic
  projection}) implies some of the simplices of $[x,y]_S$ are 
  disjoint from $W$ --- in other words the 
{\em footprint}, denoted $\phi_{[x,y]_S}(W)$ as in \S\ref{background}, 
is nonempty.

Let $W,V\in\FF$, $W\pitchfork V$,  and suppose that $\phi_{[x,y]_S}(W)$ 
is disjoint from
and to the \emph{right} of $\phi_{[x,y]_S}(V)$ 
(i.e., $\phi_{[x,y]_S}(W)$ 
lies on $[x,y]_{S}$ closer to $y$ then $\phi_{[x,y]_S}(V)$). We claim this implies
$V\prec_2 W$. Indeed, letting $t\in \phi_{[x,y]_S}(V)$, the segment
$[x,t]_S$ consists of curves intersecting $W$, and by Lemma
\ref{geodesic projection} $d_W(x,t) \le B$. Since $t$ and $\boundary
V$ are disjoint, $d_W(x,\boundary V) \le B < 2c$, so $W\not\prec_2 V$.
Since by Lemma \ref{partial order} they are $\prec_2$--ordered,
$V\prec_2 W$.  Equivalently, we can say that if 
$W\prec_2 V$, then 
$\phi_{[x,y]}(V)$ either intersects or is to the right of  
$\phi_{[x,y]}(W)$. (These are essentially variations on arguments in
\cite{MasurMinsky:complex2}.)

Now, if $U$ is $\prec_2$--minimal in $\FF$, then 
$\boundary U$  is at least $a$ from $\mu$; hence 
its footprint is at least $a-\ep$
either to the left or to the right of $q(\mu)$ (recall $d_S(\mu,q)\le
\ep$). 
If it were on the left, since $a> 2\ep$, using the Lipschitz property
of $\pi_U$ as earlier,  we would
find that $d_U(\mu,q) \le 3\ep.$ Lemma \ref{geodesic projection} would
give us $d_U(q,p)< B$, but this bounds $d_U(\mu,p)$ and contradicts
$U\in\FF$. We conclude the footprint is to the right of $q$.

By the previous paragraph on ordering, we conclude since $U$ is
$\prec_2$--minimal that 
{\em all} elements of $\FF$ have footprints at least distance
$a-\ep$ to the right of 
$q$, and in fact to the right of $p$ since $d_S(p,q) \le a/2$. 

\medskip

To get a bound on $d(\mu,\GG(p,y))$ we must again bound the terms
from the quasidistance formula,  i.e.,
$d_W(\mu,\GG(p,y))$ for $W\subseteq S$. 

Let  $W\subset S$ be a proper connected subsurface. 
If $W$ is disjoint from $p$, then $\boundary W$ is within $\CC(S)$--distance
$a/2+1$ from $q$, and hence no more than $a$ from $\mu$. 
It follows that $W$ is not in $\FF$, and hence
$d_W(\mu,y)$ is bounded. By construction, the projection of
$\GG(p,y)$ to $W$ is the projection of $y$, so this gives us the
desired bound.

If $W$ intersects $p$, we must bound $d_W(\mu,p)$. 
We claim that if $d_W(\mu,p)$ is sufficiently large, then
$\phi_{[x,y]_S}(W)$ is nonempty. For if it were empty, Lemma
\ref{geodesic projection} would bound $\diam_W([x,y]_S)$, and since
$\pi_W(\mu)$ is within $\ep$ of $[x,y]_W$, this would bound $d_W(\mu,p)$ as
well.  Hence we may assume
$\phi_{[x,y]_S}(W)$ is nonempty, and so 
lies either to the right
or the left of $p$. 

If it is on the left, then by the previous discussion $W$ cannot be in
$\FF$, and so $d_W(\mu,y)$ is bounded. Moreover
since the footprint is outside
of $[p,y]_S$, Lemma \ref{geodesic projection} gives a bound on
$d_W(\mu,p)$ as well. 

If it is on the right, then its distance from $q$ is at least $a/2$,
and so we have a bound on $d_W(\mu,q)$, by the Lipschitz property of
$\pi_W$, and on $d_W(q,p)$, by Lemma \ref{geodesic projection}. 
Hence we obtain a bound on $d_W(\mu,p)$. 

The only case left is that $W=S$. However, we have already noted that
$d_S(\mu,q) \le \ep$ and hence $d_S(\mu,p) \le \ep+ a/2$. 

This completes the proof of Lemma \ref{mu near G}.

\subsubsection*{Diameter sum bound}
Our final step is to prove Lemma \ref{G diameter sum}, 
bounding the diameter sum over $\Gamma$. 

First, consider the members of $\Gamma$ of the form $\GG(U,x,y)$. Each
of these has diameter 
comparable with $d_{\MM(U)}(\pi_{\MM(U)}(x),\pi_{\MM(U)}(y))$, which by the quasidistance formula
is estimated by
$$
\sum_{V\subseteq U}\truncate{d_{V}(x,y)}{A}
$$
for any sufficiently large $A$. Let us choose $A>2c$.
So the sum over all possible $\GG(U,x,y)$ should be comparable to the
sum of $\truncate{d_V(x,y)}{A}$ over all proper subsurfaces $V$ in $S$, provided we show
that each $V$ occurs in a bounded number of $U$'s.

Fix $V\subsetneq S$. 
Consider $\mu\in\Sigma_\ep(x,y)$ such that $\gamma(\mu) = \GG(U,x,y)$ with
$V\subset U$, and let $U_1$ be a component of $U$. 
In particular, $d_{\CC_1(S)}(\boundary U_1,\boundary V)\le 1$. 

We will control the number of possible such  $U_1$'s using a
hierarchy $H=H(x,y)$. 
Since $U_1\in\FF$ we have $d_{U_1}(\mu,y)>3c$, and hence $d_{U_1}(x,y)>3c-\ep>2c$ which 
by Lemma \ref{hierarchy facts} implies $U_1$ is a domain in $H$. 
Let $W$ be any other domain of a tight geodesic
$k=h_{x,y,W}$ in $H$ such that $U_1\subset W$. 
We claim: 
\begin{lemma}{footprint at ends}
The footprint $\phi_k(U_1)$ is a uniformly bounded distance in
$\CC_1(W)$ from one of the endpoints of $k$. 
\end{lemma}

\begin{proof}
Since the endpoints of $k$ are within $m_3$ of $\pi_W(x)$ and
$\pi_W(y)$ respectively, and $\phi_k(U_1)$ is within 1 of $\boundary
U_1$, it suffices to bound $d_W(\boundary U_1,x)$ or $d_W(\boundary
U_1,y)$.

Suppose first that $d_W(\mu,y)  \le 3c$. 
We claim in this case that $d_W(\boundary U_1,y)$ is at most $3c+2$. For
if not, then Lemma \ref{geodesic projection} can be applied to the
geodesic segment from $\pi_W(y)$ to $\pi_W(\mu)$, yielding
$d_{U_1}(\mu,y) \le B < 3c$, a contradiction with the definition of
$U_1$. Thus we are done in this case. 

Suppose now that $d_W(\mu,y)>3c$, so that $W$ is in
$\FF_3(\mu,y)$. 
It cannot be $\prec_2$--minimal,
because if it were then $U_1$ would not have been chosen ($U_1$ is
inclusion-maximal among $\prec_2$--minimal elements). Hence there is some
$W'\prec_2 W$.  If $W'\pitchfork U_1$, then they are 
$\prec_2$--ordered; but
since $U_1$ is $\prec_2$--minimal, we get $U_1\prec_2 W'$, and thus 
$U_1\prec_2 W$
which contradicts $U_1\subset W$. Hence $W'$ and $U_1$ have disjoint
boundaries, so $d_W(\boundary U_1,\boundary W') \le 3$ by the
coarse Lipschitz property of $\pi_W$. 
Now, since $W'\prec_2 W$, we have by Lemma \ref{four inequalities}
that  $d_W(x,\boundary W')$ 
is bounded, and so we get a bound on $d_W(x,\boundary U_1)$. This
completes the proof. 
\end{proof}

We now bound the number of possible $U_1$'s,  by induction. 
Using part (3) of Lemma \ref{hierarchy facts}, every $U_1$ is
contained in a chain $U_1=W_0\subset W_1 \subset\cdots\subset W_s=S$ such
that each  $W_i$ supports a geodesic in $H$, each $W_i$ for $i<s$ is a
component domain (complementary component or annulus) of a simplex in
$h_{x,y,W_{i+1}}$, and this simplex (being in the footprint) is a bounded
distance from one of the  endpoints of $h_{x,y,W_{i+1}}$ for $i+1<s$. For $i=s-1$, the
footprint is on $h_{x,y,W_s} = [x,y]_S$, and here it is constrained to
a bounded interval by the inequality $d_S(\boundary V,\boundary
U_1)\le 1$ (remember that we have fixed $V$).

Hence, starting with $W_s=S$ and working backwards, for each $W_i$
there is a uniformly bounded number of choices for $W_{i-1}$. We
conclude that there is a uniformly bounded number of choices for
$U_1=W_0$. 

\medskip

So now we have a uniform bound on the number of different
$\GG(U,x,y)$'s in $\Gamma$ for which $U$ contains a fixed subsurface
$V$. We conclude that 
\begin{equation*}
\sum_{\GG(U,x,y)\in\Gamma}
{\diam(\GG(U,x,y))} \le N \sum_{V\subsetneq S} \truncate{d_V(x,y)}{A}
\end{equation*}
for uniform $N$.

What remains in the left-hand-side of inequality (\ref{sum of diameters}) is the sum over the
$\GG(p,y)\in\Gamma$ where $p\in [x,y]_S$. Since for these, 
by definition, the diameters are all 1, this sum satisfies
\begin{equation*}
\sum_{\GG(p,y)\in\Gamma}{\diam(\GG(p,y))} \le  d_S(x,y)+1
\end{equation*}
Putting these together we have
\begin{equation*}
\sum_{\GG\in\Gamma} \diam(\GG) \le N'\sum_{V\subseteq S} \truncate{d_V(x,y)}{A}
\le N''\ d(x,y).
\end{equation*}
which establishes Lemma \ref{G diameter sum}, and so completes the
proof of Theorem \ref{polynomial bound}.
\end{proof}

\subsection{Proofs of the Main Theorems}
\label{lastsection}

To wrap up the proof of Theorem~\ref{centroid}: 
We showed that $K(x,y,r)$ is contained in $\Sigma_{\ep'}(x,y)$ and in
$\NN_{b_1r}(x)$, for uniform $\ep'$ and $b_1$. Lemma \ref{reduce to Sigma}
then implies that there exists $q\in\Sigma_\ep(x,y)$ such that
$d(x,q) \le b_2 r$, and 
$K(x,y,r)$ is contained in $\Sigma_{\ep'}(x,q)$. Finally, Theorem
\ref{polynomial bound} gives us a bound for $\#\Sigma_{\ep'}(x,q)$
which is polynomial (of degree $\xi(S)$) in $d(x,q)$. This gives the
desired bound for $\#K(x,y,r)$.

As mentioned in the introduction, Theorem~\ref{rapid decay} now follows 
by applying a result of Dru\c{t}u-Sapir 
\cite{DrutuSapir:RD}.
They showed that the Rapid Decay property holds for 
groups which are 
\emph{$(**)$--relatively hyperbolic with 
respect to the trivial group}. This property is said to hold when the 
following are satisfied: there exists a function 
$T\co G\times G\to G$ and a polynomial $Q(r)$ satisfying: 
\begin{enumerate}
    \item $T(g,h) = T(h,g)$

    \item $T(h^{-1},h^{-1}g) = h^{-1}T(h,g)$
    
    \item If $g\in G$ and $r\in\N$, then   
    $\#\{T(g,h): |h| = r\} < Q(r)$.

\end{enumerate}

For the mapping class group, we define $$T(g,h)=\kappa(1,g,h)\co 
\MCG(S)\times\MCG(S)\to \MCG(S)$$ where $\kappa$ is the 
centroid map as given by Theorem~\ref{centroid}. The first condition 
above follows from the property that $\kappa$ is invariant under 
permutation of its arguments 
(Theorem~\ref{centroid} part (\ref{centroid:perm})). The second condition holds 
since $\kappa(1,h^{-1},h^{-1}g)=h^{-1} \kappa(h,1,g) = 
h^{-1} \kappa(1,h,g)$, where the first equality is from the 
equivariance of $\kappa$ 
(Theorem~\ref{centroid} part (\ref{centroid:equivariant})). The third condition 
follows from our cardinality bound on centroids 
(Theorem~\ref{centroid} part (\ref{centroid:count})).

\medskip

Alternatively, Chatterji-Ruane in
\cite[Proposition~1.7]{ChatterjiRuane} give a similar criterion that implies
Rapid Decay. Their condition involves an equivariant family of subsets
$S(x,y)\subset G$, where $x,y\in G$, satisfying a number of conditions,
in particular
$$
S(x,y)\intersect S(y,z)\intersect S(x,z) \ne \emptyset
$$
and
$$
\#S(x,y) \le P(d(x,y))
$$
for a polynomial $P$. It is not hard to see that our $\Sigma$-hulls
give such a family, i.e., $S(x,y)\equiv \Sigma_\ep(x,y)$, and that
the existence of centroids gives the non-empty triple intersection property.

\providecommand{\bysame}{\leavevmode\hbox to3em{\hrulefill}\thinspace}

\end{document}